\documentclass[11pt]{amsart}

\usepackage{amssymb}
\usepackage{amsthm}
\usepackage[all]{xy}
\usepackage{geometry}
\usepackage{layout}

\usepackage{pxfonts}
\usepackage{newpxmath}
\usepackage{graphicx}

\newcommand{\OO}{\mathcal O}

\newcommand{\Z}{\mathbb Z}

\newcommand{\N}{\mathbb N}

\newcommand{\Hom}{\operatorname{Hom}}
\newcommand{\End}{\operatorname{End}}

\newcommand{\GL}{\operatorname{GL}}

\newcommand{\id}{\operatorname{id}}

\newcommand{\Ext}{{\operatorname{Ext}}}

\newcommand{\Inn}{\operatorname{Inn}}

\newcommand{\Aut}{\operatorname{Aut}}
\newcommand{\Out}{\operatorname{Out}}
\newcommand{\opp}{\operatorname{op}}
\newcommand{\rank}{\operatorname{rank}}
\newcommand{\UU}{\mathcal {U}}
\newcommand{\Pic}{{\operatorname{Pic}}}

\newtheorem{defi}{Definition}[section]
\newtheorem{remark}[defi]{Remark}

\newtheorem{thm}[defi]{Theorem}
\newtheorem{lemma}[defi]{Lemma}
\newtheorem{corollary}[defi]{Corollary}

\newtheorem{prop}[defi]{Proposition}
\newtheorem*{question}{Question}
\usepackage{lipsum}

\newcommand\blfootnote[1]{%
  \begingroup
  \renewcommand\thefootnote{}\footnote{#1}%
  \addtocounter{footnote}{-1}%
  \endgroup
}

\address{Florian Eisele \newline Department of Mathematics, City, University of London,
	Northampton Square, London EC1V 0HB, United Kingdom \newline E-mail address: {\tt Florian.Eisele@city.ac.uk}}

\title{The Picard group of an order and K\"ulshammer reduction}
\author{Florian Eisele}

\begin{document}
	\begin{abstract}
		Let $(K,\OO,k)$ be a $p$-modular system and assume $k$ is algebraically closed.
		We show that if $\Lambda$ is an $\OO$-order in a separable $K$-algebra, then $\Pic_\OO(\Lambda)$ carries the structure of an algebraic group over $k$. As an application to the modular representation theory of finite groups, we show that a reduction theorem by K\"ulshammer concerned with Donovan's conjecture remains valid over $\OO$.
	\end{abstract}

	\maketitle
	\blfootnote{The author is supported by EPSRC grant EP/M02525X/1.}

	\section{Introduction}

	Let $(K,\OO, k)$ be a $p$-modular system for some prime $p>0$, and assume that $k$ is algebraically closed. 
	It is well-known and easy to prove that the automorphisms group $\Aut_k(A)$ of a finite-dimensional $k$-algebra $A$ is an algebraic group. The same is true for the group of outer automorphisms $\Out_k(A)$, and therefore also for the Picard group $\Pic_k(A)$ as well as $\operatorname{Picent}(A)$, as both of these contain the identity component of $\Out_k(A)$ as a subgroup of finite index. However, not much can be said about the structure of these algebraic groups in general. The class of algebras we are interested in most in this article are blocks of finite group algebras, which have the additional property that they can be defined over both $k$ and $\OO$. In the known examples, the Picard group of a block defined over $\OO$ turns out to be much smaller than the Picard group of the corresponding block defined over $k$. For instance, the ``$F^*$-theorem'' (see \cite{HertweckKimmerle}) shows that if $N$ is a normal $p$-subgroup of a finite group $G$ such that $C_G(N) \subseteq N$, then $\Pic_\OO(\OO G)$ is finite. In the same vein, using a result of Weiss (see \cite[Theorem 2]{WeissRigidity}) one can show that if $G$ has a normal Sylow $p$-subgroup, then $\Pic_{\OO}(\OO G)$ is finite. The finiteness of the Picard group is preserved by derived equivalences, and therefore all blocks of abelian defect (defined over $\OO$) should have finite Picard groups, provided Brou\'{e}'s abelian defect group conjecture holds. 
	With all of this in mind, interest in the structure of $\Pic_\OO(\OO G)$ for finite groups $G$ has recently resurged. In particular \cite{BoltjeKessarLinckelmann} studies certain finite subgroups of $\Pic_{\OO}(\OO G)$, but leaves open the question whether these finite subgroups are proper. One of the emerging questions on Picard groups of block algebras is therefore:
	\begin{question}
		For which finite groups $G$ is $\Pic_{\OO}(\OO G)$ finite?
	\end{question}
	Of course we could ask the same question for individual blocks. To the best of the author's knowledge there are no known examples where $\Pic_{\OO}(\OO G)$ turns out to be infinite. 
	To approach the question above, one could first try to endow $\Pic_{\OO}(\OO G)$ with additional structure. That is what we do in this article, with the added benefit of an immediate application to modular representation theory. Our main result is the following structural theorem on Picard groups of more general $\OO$-orders:
	\begin{thm}[see Theorem~\ref{thm_out_algebraic}]\label{thm main 1}
		If $\Lambda$ is an $\OO$-order in a separable $K$-algebra, then $\Out_\OO(\Lambda)$ has the structure of an algebraic group over $k$.
	\end{thm}
	Note that we use the term ``algebraic group'' in the classical sense, rather than as a synonym of ``group scheme''. The question whether we can define a (possibly non-reduced) ``outer automorphism group scheme'' of an $\OO$-order is interesting, but we do not consider it in this article. 
	The proof of Theorem~\ref{thm main 1} combines the theory of Witt vectors with theorems of Maranda and Higman on lifting from $\OO/p^n \OO$ for large $n$ to $\OO$. 
	
	Theorem~\ref{thm main 1} suggests a different approach to the aforementioned finiteness question. Namely, it is sufficient to determine the Lie algebra of $\Pic_\OO(\Lambda)$, and finiteness follows if it is zero-dimensional. For a finite-dimensional $k$-algebra $A$, the Lie algebra of $\Out_k(A)$ embeds into the first Hochschild cohomology group $HH^1(A)$, which can be identified with the Lie algebra of the ``outer automorphism group scheme'' of $A$. It is tempting to ask whether a similar relationship exists between the Lie algebra of $\Pic_\OO(\Lambda)$ and $HH^1(\Lambda)$ (which vanishes if $\Lambda$ is a block). However, these questions are outside of the scope of this short article. Instead, we focus our attention on an application in modular representation theory which follows from Theorem~\ref{thm main 1} in a fairly straightforward manner, a generalisation of a theorem of K\"ulshammer (see \cite{KuelshammerDonovan}):
	\begin{thm}[see Theorem~\ref{thm kuelshammer full} and Corollary~\ref{corollary finite upstairs finite downstairs}]\label{thm kuelshammer over o}
		Let $D$ be a finite $p$-group. Then the following two statements are equivalent:
		\begin{enumerate}
		\item Donovan's conjecture for blocks of defect $D$ (defined over $\OO$): there are only finitely many Morita equivalence classes of blocks with defect group $D$.
		\item There are only finitely many Morita equivalence classes of blocks $B$ of $\OO H$ for finite groups $H$ with the property that
		\begin{enumerate}
		\item $B$ has defect group isomorphic to $D$.
		\item $H=\langle D^h \ | \ h \in H \rangle$, where we have identified $D$ with the defect group of $B$.
		\end{enumerate}
		\end{enumerate}
	\end{thm}
	Theorem~\ref{thm kuelshammer over o} was the original motivation for Theorem~\ref{thm main 1}, as the fact that the results of \cite{KuelshammerDonovan} were not known to hold over $\OO$ prevented them from being combined with the results on ``strong Frobenius numbers'' in  \cite{EatonLiveseyTowardsDonovan}, which only work over $\OO$. 
	This was one of the obstacles in reducing Donovan's conjecture for blocks of abelian defect (defined over $\OO$) to blocks of quasi-simple groups, which will be the subject of \cite{EatonEiseleLivesey} (in preparation). 

	\section{Witt vectors and $\Out_\OO(\Lambda/p^n\Lambda)$ as an algebraic group}

	\begin{defi}[{Witt vectors, cf. \cite[Chapter II \S 6]{SerreLocalFields}}]
		Let $k$ be a commutative ring.
		\begin{enumerate}
		\item Define
		$$
			W(k) = \prod_{\Z_{\geq 0}} k
		$$
		We can define addition and multiplication on $W(k)$ in such a way that $W(k)$ becomes a commutative ring, and for each $i\geq 0 $ there are polynomials 
		$\sigma_i, \mu_i \in k[X_0, \ldots, X_i, Y_0, \ldots, Y_i]$ such that
		$$
			(u+v)_i = \sigma_i(u_0,\ldots, u_i, v_0,\ldots, v_i)
		$$
		and 
		$$
			(u\cdot v)_i = \mu_i(u_0,\ldots, u_i, v_0,\ldots, v_i).
		$$
		The (injective) map 
		$$
			\tau:\ k \longrightarrow W(k):\ u \mapsto (u,0,0,\ldots)
		$$
		is multiplicative. The unit element of $W(k)$ is $(1,0,0,\ldots)$ and the zero element is $(0,0,0,\ldots)$.
		\item If $k$ is a perfect field of characteristic $p>0$, then $W(k)$ is a complete discrete valuation ring of characteristic zero with uniformiser $p$ and residue field $k$. Furthermore, the element $p=1+\ldots+1$ (the sum being taken in the ring $W(k)$) is equal to 
		$$
		(0,1,0,\ldots )
		$$
		and therefore projecting to the first $n$ components of $W(k)$ induces an isomorphism
		$$
			W(k)/p^nW(k) \stackrel{\sim}{\longrightarrow} W_n(k)
		$$
		where $W_n(k)=k^n$ denotes the ring of truncated Witt vectors of length $n$ (addition and multiplication being defined by the same polynomials $\sigma_i$ and $\mu_i$).
		\end{enumerate}
	\end{defi}
	
	For the remainder of this section let $k$ denote an algebraically closed field of characteristic $p>0$.
	
	\begin{prop}\label{prop poly is zariski}
		 Let $n,l\in \N$. If $f\in W_n(k)[X_1,\ldots, X_l]$ is a polynomial, then the map
		\begin{equation}\label{eqn poly map}
			W_n(k)^l \longrightarrow W_n(k):\ (x_1,\ldots, x_l) \mapsto f(x_1,\ldots, x_l)
		\end{equation}
		is a morphism of varieties.
	\end{prop}
	\begin{proof}
		Let us denote the map in \eqref{eqn poly map} by $\varphi_f$.
		If $f$ is a constant polynomial or equal to one of the $X_i$'s, then $\varphi_f$ is a constant function or a projection onto a direct factor, respectively. In particular, $\varphi_f$ is a morphism of varieties in this case. If $f$ and $g$ are polynomials such that $\varphi_f$ and $\varphi_g$ are morphisms of varieties, then so are $\varphi_{f+g}$ and $\varphi_{f\cdot g}$, as for any $(x_1,\ldots,x_l)\in W_n(k)^l$ the components of the image under these maps are given by
		 $$(\varphi_{f+g}(x_1,\ldots,x_l))_i=\sigma_i(\varphi_f(x_1,\ldots,x_l), \varphi_g(x_1,\ldots,x_l))$$
		 and
		 $$(\varphi_{f\cdot g}(x_1,\ldots,x_l))_i=\mu_i(\varphi_f(x_1,\ldots,x_l), \varphi_g(x_1,\ldots,x_l)).$$		 
		 Since we already showed the assertion for a set of generators of $W_n(k)[X_1,\ldots,X_l]$ we are done.
	\end{proof}
	
	\begin{prop}\label{prop inverse is zariski}
		Let $n\in \N$.
		The unit group $W_n(k)^\times$ is a Zariski-open subset of $W_n(k)$, and the inversion map $W_n(k)^\times \longrightarrow W_n(k)^\times:\ x\mapsto x^{-1}$ is a morphism of varieties.
	\end{prop}
	\begin{proof}
		$W_n(k)^\times = \{ x \in W_n(k)\ \mid\  x_0\neq 0 \}$ is clearly Zariski-open, with coordinate ring 
		$k[X_0,\ldots,X_{n-1},X_0^{-1}]$. The inversion map is given by
		\begin{equation}
			x \mapsto \tau(x_0^{-1})\cdot \sum_{i=0}^{n-1}(1+(-1)\cdot \tau(x_0^{-1})\cdot x)^i
		\end{equation}
		Multiplication (and therefore exponentiation) as well as addition are clearly given by polynomial functions in the components of the arguments. Therefore the components of the right hand side are given by polynomials in $X_0,\ldots,X_{n-1}$ and $X_0^{-1}$, which shows that the assignment is a morphism of varieties.
	\end{proof}
	
	\begin{prop}\label{prop gl0 alg grp}
		Let $n,r\in \N$. The group $\GL_r(W_n(k))$ is a Zariski-open subset of $W_n(k)^{r\times r}$, and multiplication and inversion turn $\GL_r(W_n(k))$ into an algebraic group over $k$. Every subset $X \subseteq \GL_r(W_n(k))$ given as the vanishing set of polynomials in $W_n(k)[T_{i,j}, \det((T_{i,j})_{i,j})^{-1} \ |\ i,j=1,\ldots,r]$ is Zariski-closed. 
	\end{prop}
	\begin{proof}
		By Proposition~\ref{prop poly is zariski} the polynomial $\det= \det((T_{i,j})_{i,j})$ defines a morphism of varieties $W_n(k)^{r\times r}\longrightarrow W_n(k)$. By Proposition~\ref{prop inverse is zariski} the subset $W_n(k)^\times$ of $W_n(k)$ is Zariski-open. Hence so is its preimage in $W_n(k)^{r\times r}$, which is $\GL_r(W_n(k))$. It also follows from Proposition~\ref{prop inverse is zariski} that $\det^{-1}$ defines a morphism of varieties $\GL_r(W_n(k))\longrightarrow W_n(k)^\times$. By looking at multiplication and inversion maps entry by entry they are now easily seen to be morphisms (using Proposition~\ref{prop poly is zariski}). The same is true for the claim on subsets defined by polynomials.
	\end{proof}
	
	\begin{prop}\label{prop aut alg}
		Let $n\in\N$ be arbitrary and
		let $\Lambda$ be a $W_n(k)$-algebra which is free and finitely generated as a $W_n(k)$-module. 
		\begin{enumerate}
		\item $
			\Aut_{W_n(k)}(\Lambda)
		$ 
		is an algebraic group over $k$.
		\item $\Inn(\Lambda)$ is a Zariski-closed subgroup of $\Aut_{W_n(k)}(\Lambda)$. In particular, the quotient $$\Out_{W_n(k)}(\Lambda)=\Aut_{W_n(k)}(\Lambda)/\Inn(\Lambda)$$ is an algebraic group. 
		\end{enumerate}
	\end{prop}
	\begin{proof}
		Let $r=\rank_{W_n(k)}(\Lambda)$ and choose a $W_n(k)$-basis 
		$\lambda_1,\ldots,\lambda_r$ of $\Lambda$. There are structure constants $c_{i,j;v}\in W_n(k)$ for $i,j,v\in \{1,\ldots,r\}$ such that $\lambda_i\cdot \lambda_j= \sum_{v=1}^rc_{i,j;v}\cdot \lambda_v$.
		\begin{enumerate}
		\item The group $\Aut_{W_n(k)}(\Lambda)$ is equal to
		$$\resizebox{0.85\hsize}{!}{$
			\left\{\ (m_{i,j})_{i,j} \in \GL_r(W_n(k))
			\ \Bigg| \ \sum_{s,t=1}^r m_{i,s}\cdot m_{j,t}\cdot c_{s,t;v}= \sum_{w=1}^r c_{i,j;w} \cdot m_{w,v}
			 \textrm{ for } i,j,v=1,\ldots,r\ \right\}
		$}$$
		which is an algebraic group by Proposition~\ref{prop gl0 alg grp}.
		\item The action of $W_n(k)^{r\times r}$ on $\Lambda$ with respect to the basis $\lambda_1,\ldots,\lambda_r$ gives us two different embeddings $\Lambda \longrightarrow W_n(k)^{r\times r}$ corresponding to right and  left multiplication, respectively. Note that $\UU (\Lambda)$ is simply the preimage of $\GL_r(W_n(k))$ under either of those embeddings. Hence $\UU(\Lambda)$ is a Zariski-open subset of $\Lambda$, and we get two morphisms of varieties
		$\UU(\Lambda)\longrightarrow \GL_r(W_n(k))$. One of those is already a homomorphism of groups, the other one becomes one after being composed with the inversion map on $\GL_r(W_n(k))$. Now we just consider the diagonal embedding of $\UU(\Lambda)$ into $\UU(\Lambda)\times \UU(\Lambda)$ followed by the direct product of the two group homomorphisms into $\GL_r(W_n(k))$, followed by multiplication $\GL_r(W_n(k))\times \GL_r(W_n(k))\longrightarrow \GL_r(W_n(k))$. All of the maps involved in this construction are morphisms of varieties, and the resulting map is a group homomorphism $\UU(\Lambda)\longrightarrow \Aut_{W(k)}(\Lambda)$ whose image is $\Inn(\Lambda)$. This proves the claim, as the image of a homomorphism of affine algebraic groups is automatically closed.  \qedhere
		\end{enumerate}
	\end{proof}

	\begin{prop}\label{prop ext}
		Let $\OO$ be a commutative $W(k)$-order, and let $\Lambda$ be an $\OO$-order. Then $\GL_\OO(\Lambda/p^n\Lambda)\subseteq \GL_{W(k)}(\Lambda/p^n\Lambda)$ is an algebraic subgroup for any $n\in \N$. In particular
		$$\Aut_\OO(\Lambda/p^n \Lambda)=\Aut_{W(k)}(\Lambda/p^n\Lambda)\cap \GL_\OO(\Lambda/p^n\Lambda)$$ is an algebraic group, and so is 
		$$\Out_\OO(\Lambda/p^n \Lambda)=\Aut_\OO(\Lambda/p^n \Lambda)/\Inn(\Lambda/p^n\Lambda)$$  
	\end{prop}
	\begin{proof}
		Set $l=\rank_{W(k)}(\OO)$ and $r=\rank_\OO(\Lambda)$. The ring
		$\OO$ acts $W(k)$-linearly on $\Lambda$ by multiplication, which induces a homomorphism $\OO\longrightarrow \End_{W(k)}(\Lambda/p^n\Lambda)\cong W_n(k)^{r\cdot l\times r\cdot l}$. Let $c_1,\ldots,c_l\in W_n(k)^{r\cdot l\times r\cdot l}$ denote the images of a $W(k)$-basis of $\OO$ under this homomorphism. Then 
		$$\resizebox{0.95\hsize}{!}{$
			\GL_\OO(\Lambda/p^n\Lambda) = \left\{ (m_{i,j})_{i,j} \in \GL_{r\cdot l}(W_n(k)) \ \Bigg| \ \sum_{t=1}^{r\cdot l} (c_a)_{s,t} \cdot m_{t,v} = \sum_{t=1}^{r\cdot l} m_{s,t}\cdot (c_a)_{t,v} \textrm{ for all $1\leq s,v\leq r\cdot l$ and $1\leq a \leq l$} \right\}
		$}$$
		which is an algebraic group by Proposition~\ref{prop gl0 alg grp}. The other assertions follow immediately.
	\end{proof}

	\begin{corollary}\label{corollary out mod pn alg}
		Let $\OO$ be a commutative $W(k)$-order, and let $\Lambda$ be an $\OO$-order.  For any $n\in \N$
		the group $\Out_{\OO}(\Lambda/p^n\Lambda)$ is an affine algebraic group (as we saw), and for any $m \geq n$ the canonical  map $$\Out_{\OO}(\Lambda/p^m\Lambda)\longrightarrow \Out_{\OO}(\Lambda/p^n\Lambda)$$ is a homomorphism of algebraic groups.
	\end{corollary}
	\begin{proof}
		This follows immediately from Propositions \ref{prop aut alg} and \ref{prop ext} (take the morphism $\GL_r(W_m(k))\longrightarrow \GL_r(W_n(k))$ given by truncating the entries, restrict to $\Aut_{\OO}(\Lambda/p^m\Lambda)$ and then pass to the quotient by inner automorphisms).
	\end{proof}

	\section{$\Out_\OO(\Lambda)$ as a subgroup of $\Out_\OO(\Lambda/p^n\Lambda)$}
	
	In this section we let $(K,\OO,k)$ denote an arbitrary $p$-modular system. We assume that $\OO$ is complete, and by $\pi$ we denote a generator of the unique maximal ideal of $\OO$.
	Note that $\OO$ is a commutative $W(k)$-order.
	
	\begin{defi}
		Let $\Lambda$ be an $\OO$-order in a separable $K$-algebra. Let $I(\Lambda) \unlhd \OO$ be the ideal consisting of those elements which annihilate $H^1(\Lambda, T)$ for all two-sided $\Lambda$-modules $T$. By \cite{HigmanIsoOrders} the ideal $I(\Lambda)$ is non-zero, and therefore equal to $\pi^{d(\Lambda)} \OO$ for some $d(\Lambda)\in \N$. We call $d(\Lambda)$ the \emph{depth of $\Lambda$}.
	\end{defi}
	
	\begin{remark}
	The cohomology groups $H^i(\Lambda, -)$ above are taken in the sense of Hochschild cohomology (see \cite[Chapter 9]{Weibel} or \cite{CartanEilenberg}). 
	\begin{enumerate}
	\item $\pi^{d(\Lambda)} \cdot H^1(\Lambda, T)=0$ for all $\Lambda$-$\Lambda$-bimodules $T$ implies that $\pi^{d(\Lambda)} \cdot H^i(\Lambda, T)=0$ for all $i\geq 1$ and all $\Lambda$-$\Lambda$-bimodules $T$. (see \cite[Introduction]{HigmanIsoOrders})
	\item 
	If $M$ and $N$ are $\Lambda$-lattices then 
	$H^1(\Lambda, \Hom_\OO(M,N)) \cong \Ext_{\Lambda}^1(M,N)$.
	(see \cite[Lemma 9.1.9]{Weibel})
	\item The depth of $\OO G$ is equal to the $\pi$-valuation of $|G|$.
	(see \cite[Introduction]{HigmanIsoOrders})
	\end{enumerate}
	\end{remark}
	
	\begin{lemma}\label{lemma maranda}
		The canonical map 
		$$
			\Out_\OO(\Lambda)\longrightarrow \Out_\OO(\Lambda/\pi^{s}\Lambda)
		$$
		is injective for any $s \geq d(\Lambda\otimes_{\OO}\Lambda^{\opp})+1$.
	\end{lemma}
	\begin{proof}
		Our claim is equivalent to the assertion that an isomorphism of 
		$\Lambda$-$\Lambda$-bimodules
		$$
			_{\alpha} (\Lambda/\pi^s\Lambda) \cong {_{\beta} (\Lambda/\pi^s\Lambda) }
		$$
		for automorphisms $\alpha,\beta \in \Aut_\OO(\Lambda)$ implies the existence of an isomorphism 
		$$
			_{\alpha} \Lambda \cong {_{\beta} \Lambda  }
		$$
		Hence it suffices to prove that two $\Lambda$-$\Lambda$-bilattices 
		$M$ and $N$ are isomorphic if and only if $M/\pi^{s}M\cong N/\pi^{s}N$. We may regard $\Lambda$-$\Lambda$-bimodules as $\Lambda\otimes_{\OO}\Lambda^{\opp}$-modules. Hence an application of Maranda's theorem \cite[Theorem (30.14)]{CurtisReinerI} implies the claim.
	\end{proof}
	
	
	\begin{lemma}[cf. \cite{HigmanIsoOrders}]\label{lemma higman}
		Let $\Lambda$ be an $\OO$-order in a separable $K$-algebra. Then the image of the canonical map
		$$
			\Aut_\OO(\Lambda) \longrightarrow \Aut_\OO(\Lambda/\pi^{s+1}\Lambda)
		$$
		is equal to the image of the canonical map
		$$
			\Aut_\OO(\Lambda/\pi^{2s+1}\Lambda) \longrightarrow \Aut_\OO(\Lambda/\pi^{s+1}\Lambda)
		$$
		for all $s\geq d(\Lambda)$
	\end{lemma}
	\begin{proof}
		The proof from \cite{HigmanIsoOrders} carries over to this situation almost verbatim. We still provide a complete proof here, as the proof found in \cite{HigmanIsoOrders} technically proves a different statement and contains several misprints.
	
		Start with an element $\beta \in \Aut_\OO(\Lambda/\pi^{2s+1}\Lambda)$. We can find an $\OO$-linear map $\alpha_1:\ \Lambda\longrightarrow \Lambda$ which reduces to $\beta$ mod $\pi^{2s+1}$. We need to show that there is an algebra automorphism $\alpha$ of $\Lambda$ such that $\alpha\equiv \alpha_1 \ (\textrm{mod }\pi^{s+1})$.

		We will proceed by induction. 
		Namely, we will show that if $\alpha_i:\ \Lambda \longrightarrow \Lambda$ is an $\OO$-linear map 
		inducing an algebra automorphism on $\Lambda/\pi^{2s+i}\Lambda$ for some $i \geq 1$, then there is an an $\OO$-linear map $\alpha_{i+1}:\ \Lambda \longrightarrow \Lambda$ inducing an algebra automorphism of $\Lambda/\pi^{2s+i+1}\Lambda$ such that 
		\begin{equation}\label{converge}
		\alpha_{i+1} \equiv \alpha_{i} \ (\textrm{mod }\pi^{s+i}).
		\end{equation}
		Once we have shown this, we can simply take $$\alpha = \lim_{i\rightarrow \infty} \alpha_i$$
		Equation~\eqref{converge} shows that this is a convergent series, and that the reduction mod $\pi^{s+1}$ of $\alpha$ is exactly $\beta$. Moreover, $\alpha$ is a bijective since $\beta$ is bijective (bijectivity can be checked mod $\pi$, thanks to the Nakayama lemma). It is an automorphism since each  $\alpha_i$ induces an automorphism on $\Lambda/\pi^{s+i}\Lambda$.
		
		Now let us find $\alpha_{i+1}$. Define an $\OO$-linear map 
		\begin{equation}
			f:\ \Lambda\otimes_{\OO}\Lambda \longrightarrow \Lambda: \ 
			x\otimes y \mapsto \alpha_i(xy)-\alpha_i(x)\cdot \alpha_i(y).
		\end{equation}
		Then $f$ satisfies the relation
		\begin{equation}\label{eqn cocycle}
			f(x\otimes yz)-f(xy\otimes z)=f(x\otimes y) \cdot \alpha_i(z)
			- \alpha_i(x) \cdot f(y\otimes z)
		\end{equation}
		Because $\alpha_i$ induces, by assumption, an automorphism mod $\pi^{2s+i}$, it follows that the image of $f$ is contained in 
		$\pi^{2s+i}\Lambda$, that is, $f=\pi^{2s+i}\cdot g$ for some $\OO$-linear map $g:\ \Lambda \otimes_\OO \Lambda\longrightarrow \Lambda$, subject to the same identity as $f$ (i.e. \eqref{eqn cocycle}). Since $\alpha_i$ is invertible we can define the map
		$$
			\bar g:\ \Lambda\otimes_\OO \Lambda \longrightarrow \Lambda/\pi^{2s+i}\Lambda: \ x\otimes y \mapsto \alpha_i^{-1}(g(x\otimes y)) + \pi^{2s+i}\Lambda
		$$
		Equation~\eqref{eqn cocycle} yields
		\begin{equation}\label{eqn cocycle2}
			\bar g(x\otimes yz)-\bar g (xy\otimes z)=\bar g(x\otimes y) \cdot z
			- x \cdot \bar g (y\otimes z)
		\end{equation}		
		which implies that $\bar g$ defines an element of $H^2(\Lambda, \Lambda/\pi^{2s+i}\Lambda)$. By assumption, $\pi^s$ annihilates this cohomology group, that is,
		\begin{equation}
			\pi^s\cdot \bar g (x\otimes y) = x\cdot h(y)+h(x)\cdot y -h(x\cdot y) + \pi^{2s+i}\Lambda
		\end{equation} 	
		for some $\OO$-linear map $h:\ \Lambda\longrightarrow \Lambda$. 
		Define
		\begin{equation}
			\alpha_{i+1}(x) = \alpha_i(x) + \pi^{s+i} \cdot \alpha_i(h(x)).
		\end{equation}
		We have
		\begin{equation}
			\begin{array}{rcl}
				\alpha_{i+1}(x\cdot y) &=& \alpha_i(x\cdot y)+ \pi^{s+i}\cdot  \alpha_i(h(x\cdot y))\\
				&=& \alpha_i(x)\cdot \alpha_i(y)+f(x\otimes y)+\pi^{s+i} \cdot \alpha_i(h(x\cdot y)) \\
				&=&\alpha_i(x) \cdot\alpha_i(y)+ \pi^{s+i}\cdot (\pi^s\cdot g(x\otimes y) + \alpha_i(h(x\cdot y)))
			\end{array}
		\end{equation}
		Now 
		\begin{equation}
			\begin{array}{rcll}
			\pi^s\cdot g(x\otimes y) + \alpha_i(h(x\cdot y)) 
			&\equiv& \alpha_i(\pi^s \cdot \bar g (x\otimes y) + h(x\cdot y))
			  \\
			 &\equiv& \alpha_i(x\cdot h(y)+h(x)\cdot y)
			 			 & (\textrm{mod }\pi^{2s+i})
			 \end{array}
		\end{equation}
		Since $2\cdot (s+i) \geq 2s+i+1$ we get
		\begin{equation}
			\begin{array}{rcll}
				\alpha_{i+1}(x\cdot y)  &\equiv& \alpha_i(x)\cdot \alpha_i(y)+\pi^{s+i} \cdot (\alpha_i(x\cdot h(y)+h(x)\cdot y))  \\
				&\equiv& (\alpha_i(x)+\pi^{s+i}\cdot \alpha_i(h(x)))\cdot  
				(\alpha_i(y)+\pi^{s+i}\cdot \alpha_i(h(y))) \\
				&\equiv& \alpha_{i+1}(x)\cdot \alpha_{i+1}(y) & (\textrm{mod } \pi^{2s+i+1})
			\end{array}
		\end{equation}
		That is, $\alpha_{i+1}$ induces an algebra automorphism on $\Lambda/\pi^{2s+i+1}$ as claimed.
	\end{proof}
	
	\begin{thm}\label{thm_out_algebraic}
		Let $\Lambda$ be an $\OO$-order in a separable $K$-algebra, and let  $s\geq \max\{ d(\Lambda), d(\Lambda \otimes_\OO \Lambda^{\opp}) \}$.
		Then the canonical map 
		\begin{equation}\label{eqn thm1}
			\Out_\OO(\Lambda) \longrightarrow \Out_\OO(\Lambda/\pi^{s+1}\Lambda)
		\end{equation}
		is injective and its image is equal to the image of the canonical map
		\begin{equation}\label{eqn thm2}
			\Out_\OO(\Lambda/\pi^{t}\Lambda) \longrightarrow \Out_\OO(\Lambda/\pi^{s+1}\Lambda).
		\end{equation}
		for any $t \geq 2s+1$.
		In particular, if  $k$ is algebraically closed and $r\in \N$ is chosen such that $\pi^r \OO=p\OO$, then $\Out_\OO(\Lambda)$ is isomorphic to an algebraic subgroup of the algebraic group $\Out_\OO(\Lambda/p^{n}\Lambda)$, for any natural number $n \geq \frac{s+1}{r}$.
	\end{thm}
	\begin{proof}
		The injectivity of the map in equation~\eqref{eqn thm1} follows immediately from Lemma~\ref{lemma maranda}. The fact that the images of the maps in equations \eqref{eqn thm1} and \eqref{eqn thm2} coincide follows from Lemma~\ref{lemma higman}.
		The claim that the morphism given in \eqref{eqn thm2} is a morphism of algebraic groups (making its image an algebraic group) follows from Corollary~\ref{corollary out mod pn alg}.
	\end{proof}

	\section{K\"ulshammer reduction}
	
	Throughout this section let $(K,\OO,k)$ be a $p$-modular system for some $p>0$ and assume that $k$ is algebraically closed.
	In \cite{KuelshammerDonovan}, K\"ulshammer showed that any block 
	$B$ of a group algebra $kG$ for a finite group $G$ is Morita equivalent to a \emph{crossed product}
	of a block $B'$ of $kH$ and a $p'$-group $X$, where $H$ is a finite group and $B'$ has the property that its defect groups generate $H$.
	The order of the $p'$-group $X$ is also bounded in terms of $D$. Moreover, K\"ulshammer showed that there are only finitely many isomorphism classes of crossed products between any given finite-dimensional $k$-algebra $A$ and any given $p'$-group $X$. 
	This last result is the only part of \cite{KuelshammerDonovan} where the arguments rely on the coefficient ring being a field, while all other results carry over to blocks over $\OO$ with minimal effort. The main ingredient in the proof that there are only finitely many crossed products of $A$ and $X$ is the fact that there are only finitely many conjugacy classes of homomorphisms $X\longrightarrow \Out_k(A)$ on account of $\Out_k(A)$ being an algebraic group over $k$. Now Theorem~\ref{thm_out_algebraic} shows that $\Out_\OO(\Lambda)$ is an algebraic group as well, where $\Lambda$ is an $\OO$-order in a separable $K$-algebra. Hence one would expect that K\"ulshammer's results remain true if $k$ is replaced by $\OO$. That is what we show in this section, closely following the line of reasoning from \cite{KuelshammerDonovan}.
	
	\begin{defi}
		Let $G$ be a finite group.
		\begin{enumerate}
		\item A $G$-graded ring $A=\bigoplus_{g\in G} A_g$ is called a \emph{crossed product of $R=A_1$ and $G$} if $A_g\cap \UU(A) \neq \emptyset$ for all $g\in G$.
		
		\item We call two $G$-graded rings $A$ and $B$ with $A_1=B_1=R$ \emph{weakly equivalent} if they are isomorphic as $G$-graded rings, that is, if there is a ring isomorphism $\varphi:\ A \longrightarrow B$ such that $\varphi(A_g)=B_g$ for all $g\in G$. 
		
		\item Let $R$ be a ring. A \emph{parameter set} of $G$ in $\Lambda$ is a pair $(\alpha,\gamma)$ such that 
		$\alpha:\ G\longrightarrow \Aut_{\Z} (R):\ g \mapsto \alpha_g$ and $\gamma:\ G\times G \longrightarrow \mathcal U(R)$ are maps satisfying the identities
		\begin{equation}\label{eqn id alpha gamma}
			\alpha_g\circ \alpha_h= \iota_{\gamma(g,h)}\circ \alpha_{gh} \quad \textrm{and} \quad \gamma(g,h)\cdot \gamma(gh,k)=\alpha_g(\gamma(h,k))\cdot \gamma(g,hk)
		\end{equation}
		for all $g,h,k\in R$, where $\iota_{\gamma(g,h)}$ denotes the inner automorphism of $R$ induced by conjugation by $\gamma(g,h)$.
		
		\item Two parameter sets $(\alpha,\gamma)$ and $(\alpha',\gamma')$ are called \emph{equivalent} if there is a map $r:\ G\longrightarrow \mathcal U(R)$ such that
		\begin{equation}\label{trafo param sys}
			\alpha'_g= \iota_{r(g)}\circ \alpha_g \quad\textrm{and}\quad \gamma'(g,h)=r(g)\cdot  \alpha_g(r(h))	\cdot \gamma(g,h) \cdot r(gh)^{-1}		
		\end{equation}
		for all $g,h\in G$.
		\end{enumerate}
	\end{defi}
	
	\begin{defi}\label{defi crossed product}
	A parameter set $(\alpha, \gamma)$ defines a crossed product 
	$A=R*_{(\alpha,\gamma)}G$. We construct $A$ as a free $G$-graded left $R$-module with basis $\{u_g\ | \ g\in G\}$ such that the homogeneous component $A_g$ is equal to $R u_g$. Multiplication is defined by the rule
	\begin{equation}\label{eqn mult in crossed product}
		r_g u_g \cdot r_hu_h = r_g\alpha_g(r_h)\gamma(g,h) u_{gh} \textrm{ for all $r_g,r_h\in R$ }
	\end{equation}
	extended additively and distributively.
	\end{defi}

	\begin{lemma}[{see \cite[Section 4]{KuelshammerDonovan}}]
	\begin{enumerate}
		\item If $(\alpha,\gamma)$ and $(\alpha',\gamma')$ are equivalent parameter sets, then the induced homomorphisms $\bar \alpha:\ G \longrightarrow \Out_{\Z}(R)$ and $\bar \alpha':\ G \longrightarrow \Out_{\Z}(R)$ are equal.
		\item The parameter sets $(\alpha,\gamma)$ such that $\alpha$ induces a given homomorphism $\omega:\ G \longrightarrow \Out_{\Z}(R)$ are in bijection with the set $H^2(G, {^\omega\UU(Z(R))})$, where ${^\omega\UU(Z(R))}$ denotes the $\Z G$-module which is equal to $\UU (Z(R))$ as an abelian group, the action of $G$ being given by $\omega|_{Z(R)}\in \Out_{\Z}(Z(R))=\Aut_{\Z}(Z(R))$.
	\end{enumerate}
	\end{lemma}
	
	\begin{lemma}
	If $G$ is a finite group of order co-prime to $p$ and $\Lambda$ is a commutative $\OO$-order on which $G$ acts by $\OO$-algebra automorphisms, then $H^i(G, {\UU (\Lambda)})$ is finite for all $i\geq 1$.
	\end{lemma}
	\begin{proof}
		We have a short exact sequence
		$
			0\longrightarrow {(1+J(\Lambda))}\longrightarrow {\UU (\Lambda)} \longrightarrow { \UU(\Lambda/J(\Lambda))} \longrightarrow 0
		$, where $J(\Lambda)$ denotes the Jacobson radical of $\Lambda$. The resulting long exact sequence of cohomology yields exact sequences $H^i(G, 1+J(\Lambda))\longrightarrow H^i(G, \UU(\Lambda)) \longrightarrow H^i(G, \UU(\Lambda/J(\Lambda)))$ for all $i \geq 1$. Hence it suffices to show that $H^i(G, 1+J(\Lambda))$ and $H^i(G, \UU(\Lambda/J(\Lambda)))$ are finite. As $\Lambda/J(\Lambda)$ is a finite-dimensional $k$-algebra, the latter follows immediately from \cite[Proposition in section 3]{KuelshammerDonovan}. And $H^i(G, 1+J(\Lambda))$ is in fact zero, as the map $1+J(\Lambda)\longrightarrow 1+J(\Lambda):\ 1+x\mapsto (1+x)^{|G|}=1+|G|\cdot x+(\cdots)\cdot x^2$ is bijective (surjectivity follows from Hensel's lemma, and if $x$ lies in the kernel, then $1+x^2\Lambda=(1+x)^{|G|} + x^2 \Lambda=1+|G|\cdot x + x^2\Lambda$, i.e. $|G|\cdot x \in x^2\Lambda$, which is impossible).
	\end{proof}
	
	\begin{defi}
		Let $S$ be a commutative ring, let $R$ be an $S$-algebra and let $G$ be a finite group.
		\begin{enumerate}
		\item Let $A$ be a crossed product of $R$ and $G$.
		We say that $A$ is an \emph{$S$-linear crossed product} if the image of $S$ under the canonical embedding $R\hookrightarrow A$ is contained in $Z(A)$ (turning $A$ into an $S$-algebra and the natural embedding into an $S$-algebra homomorphism).
		\item We call a parameter set $(\alpha,\gamma)$ of $G$ in $R$ an \emph{$S$-linear parameter set} if the image of $\alpha$ is contained in $\Aut_S(R)$ (rather than just $\Aut_{\Z}(R)$).  
		\end{enumerate}
	\end{defi}
	
	\begin{prop}\label{prop s lin param}
		Assume $R$ is an $S$-algebra for some commutative ring $S$.
		The crossed product $R*_{(\alpha,\gamma)}G$ is $S$-linear if and only if the parameter set $(\alpha,\gamma)$ is $S$-linear.
	\end{prop}
	\begin{proof}
		We first show that we can assume without loss that $\gamma(1,1)=1$. This is achieved by replacing $(\alpha,\gamma)$
		by an equivalent parameter system $(\alpha',\gamma')$, defined via formula \eqref{trafo param sys} with $r(g)=\gamma(1,1)^{-1}$ for all $g$. By formula \eqref{eqn id alpha gamma} for $g=h=1$, $\alpha_1=\iota_{\gamma(1,1)}$, so $\alpha_1(\gamma(1,1))=\gamma(1,1)$.
		Now \eqref{trafo param sys} evaluated at $g=h=1$ with our choice of $r$ yields $\gamma'(1,1)=1$.
		So assume $\gamma(1,1)=1$. It follows that $\alpha_1=\iota_{\gamma(1,1)}=\id_R$. By setting $g=h=1$ in \eqref{eqn id alpha gamma} we obtain $\gamma(1,k)=1$ for all $k\in G$ and by setting $h=k=1$ we obtain $\gamma(g,1)=1$ for all $g\in G$.
		
		All we need to show to prove our claim is that the embedding $R \hookrightarrow R*_{(\alpha,\gamma)} G$ maps $S$ into the centre of 
		$R*_{(\alpha,\gamma)}G$ if and only if $\alpha_g \in \Aut_S(R)$, or, equivalently, $\alpha_g|_S=\id_S$  for all $g\in G$. 
		$S$ gets mapped into the centre of $R*_{(\alpha,\gamma)}G$ if and only if for all $s\in S$, $r\in R$ and $g \in G$ we have
		\begin{equation}
			su_1\cdot r u_g = r u_g\cdot su_1
			\Leftrightarrow
			s\alpha_1(r)\gamma(1,g)u_g = r\alpha_g(s)\gamma(g,1) u_g \Leftrightarrow sr=r\alpha_g(s)
		\end{equation}
		As $S$ is central in $R$, the rightmost equation is satisfied for all $r$ and $s$ if $\alpha_g|_S = \id_S$. In the other direction, setting $r=1$ implies $\alpha_g|_S=\id_S$.
	\end{proof}
	
	\begin{defi}
		Let $S$ be a commutative ring and let $R$ be an $S$-algebra. 
		\begin{enumerate}
		\item If $A$ and $B$ are $S$-linear crossed products of $R$ and $G$, then we say $A$ and $B$ are \emph{weakly equivalent} if $A$ and $B$ are isomorphic as $G$-graded $S$-algebras.
		\item We define an action of $\Aut_S(R)$ on $S$-linear parameter sets as follows: for $\tau\in\Aut_S(R)$ and an $S$-linear parameter set $(\alpha,\gamma)$, define ${^\tau}(\alpha,\gamma)=({^\tau\alpha},{^\tau\gamma})$, where
		$({^\tau}\alpha)_g= \tau\circ \alpha_g\circ \tau^{-1}$ and ${^\tau\gamma}=\tau\circ\gamma$.
		\end{enumerate}
	\end{defi}
	Note that the action of $\Aut_S(R)$ on equivalence classes of $S$-linear parameter sets is well defined.
	\begin{prop}
		Let $S$ be a commutative ring, and let $R$ be an $S$-algebra. There is a bijection between weak equivalence classes of $S$-linear crossed products of $R$ and $G$ and $\Aut_S(R)$-orbits on $S$-linear parameter sets of $G$ in $R$.
	\end{prop}
	\begin{proof}
			Given an $S$-linear parameter set $(\alpha,\gamma)$ and 
			a $\tau\in \Aut_S(R)$, it follows from \eqref{eqn mult in crossed product} that the assignment $r_gu_g \mapsto \tau (r_g)v_g$ defines
			an isomorphism of $G$-graded $S$-algebras between $R*_{(\alpha,\gamma)}G = \bigoplus Ru_g$ and
			$R*_{^\tau(\alpha,\gamma)}G = \bigoplus Rv_g$. In the other direction, given such an isomorphism, we may take $\tau$ to be the restriction to $R$ of this isomorphism (after normalising $(\alpha,\gamma)$ in such a way that $\gamma(1,1)=1$, as we did in Proposition~\ref{prop s lin param}, which turns $r\mapsto ru_1$ and $r\mapsto rv_1$ into morphisms of $S$-algebras).
	\end{proof}
	
	\begin{corollary}\label{corollary finitely many crossed prod}
		 If $\Lambda$ is an $\OO$-order in a separable $K$-algebra, and $G$ is a finite group, then there are only finitely many $\OO$-linear crossed products of $\Lambda$ and $G$.
	\end{corollary}
	
	If $G$ is a finite group and $N\unlhd G$ is a normal subgroup, then 
	$\OO G$ is a crossed product of $\OO N$ and $X=G/N$. To be more precise: if we fix a map $X\longrightarrow G: x \mapsto [x]$ assigning coset representatives, then the maps $\alpha_x: \OO N \longrightarrow \OO N: a \mapsto [x]\cdot a \cdot [x]^{-1}$ for $x\in X$ and $\gamma(x,y)=[x]\cdot [y]\cdot [xy]^{-1}$ for $x,y \in X$ define an $\OO$-linear parameter set such that $\OO G \cong \OO N *_{(\alpha,\gamma)} X$. This descends to blocks in some cases, as the following corollary shows.
	\begin{corollary}\label{corollary crossed prod}
		If $G$ is a finite group, $N\unlhd G$ is a normal subgroup and $b\in Z(\OO G)\cap \OO N$ is an idempotent, then 
		$\OO G b$ is an $\OO$-linear crossed product of $\OO N b$ and $G/N$.
	\end{corollary} 
	\begin{proof}
	If we write $\OO G$ as $\OO N *_{(\alpha,\gamma)} G/N$, then
	one readily sees that $\OO G b \cong 
		\OO N b *_{(\alpha',\gamma')}X$ with $\alpha'_x=\alpha_x|_{\OO N b}$ and 
		$\gamma'(x,y)=\gamma(x,y)\cdot b$.
	\end{proof}
	
	\begin{prop}[Fong correspondence]\label{prop fong}
		Let $G$ be a finite group and and let $N\unlhd G$ be a normal subgroup. Let $B=\OO G b$ be a block of $\OO G$ and let $C=\OO N c$ be a block of $\OO N$ such that $b\cdot c \neq 0$. Set $T=C_G(c)$. Then there is a unique block $E=\OO T e$ of $\OO T$ such that $b \cdot c\cdot e \neq 0$. The block $E$ is called the \emph{Fong correspondent} of $B$ with respect to $C$. It has the following properties:
		\begin{enumerate}
		\item $\OO G e$ is a free right $E$-module of rank $[G:T]$.
		\item The left action of $B$ on $\OO G e$ induces an isomorphism between $B$
		and $\End_{E}(\OO G e) \cong E^{[G:T]\times [G:T]}$.
		\end{enumerate} 
	\end{prop}
	\begin{proof}
		Note that the definition of $T$ and $E$ does not depend on whether the coefficient ring is a discrete valuation ring or a field. Therefore, the uniqueness of $E$ follows directly from \cite[Theorem C]{KuelshammerpBlocks}. Moreover, \cite[Theorem C]{KuelshammerpBlocks}
		shows that $kG e$ is a free right $kTe$-module of rank $[G:T]$, which immediately implies that $\OO G e$ is a free $\OO T e$-module of the same rank, as projective modules are determined by their reductions modulo $\pi$ (see \cite[Theorems (30.4) and (30.11)]{CurtisReinerI}). As $\OO Ge$ is projective as a right $E$-module, we have
		$k\otimes \End_{E}(\OO G e) \cong \End_{kTe}(kGe)$ (the map being the canonical one). In particular, the map $B \longrightarrow \End_{E}(\OO G e)$ induced by left multiplication becomes an isomorphism upon tensoring with $k$. Hence the Nakayama lemma implies that the map itself has to be an isomorphism.
	\end{proof}
	
	\begin{defi}\label{defi nat Morita}
		Let $R\in \{\OO, k\}$, let $G$ be a finite group and let $H\subseteq G$ be a subgroup. If $B$ is a block of $R G$ and $C$ is a block of $R H$, then $B$ and $C$ are called \emph{naturally Morita equivalent} if 
		there is an $R$-subalgebra $S \subseteq B$ such that $S\cong R^{n\times n}$ for some $n\in \N$ and multiplication induces an isomorphism
		$C\otimes_R S \longrightarrow B$.
	\end{defi}
	
	\begin{prop}[{\cite[Proposition 2.4]{HidaKoshitani}}]\label{prop Hida Koshitani}
		In the setting of Definition~\ref{defi nat Morita}, two blocks $\OO G b$ and $\OO H c$ are naturally Morita equivalent if and only if $kG b$ and $kHc$ are naturally Morita equivalent. 
	\end{prop}
	
	\begin{corollary}[{\cite[Theorem 7]{KuelshammerAsterisque}}]\label{corollary nat morita}
		Let $G$ be a finite group and let $N\unlhd G$ be a normal subgroup. Let $B=\OO G b$ and $C=\OO N c$ be blocks such that $b\cdot c \neq 0$ (i.e. $B$ covers $C$). Then $B$ and $C$ are naturally Morita equivalent if and only if they share a defect group and $C_G(c)=G$ (i.e. the block $C$ is stable under the action of $G$ by conjugation).
	\end{corollary}
	\begin{proof}
		The analogue of the assertion over $k$ is exactly the statement of \cite[Theorem 7]{KuelshammerAsterisque}. Now Proposition~\ref{prop Hida Koshitani} allows us to replace $k$ by $\OO$.
	\end{proof}
	
	\begin{prop}\label{prop crossed prod condens}
		Let $\Lambda$ be an $\OO$-order, let $G$ be a finite group and let $\Gamma$ be an $\OO$-linear crossed product of $\Lambda$ and $G$. If $e\in \Lambda$ is an idempotent with the property that $e\Lambda \cong \alpha(e)\Lambda$ as right $\Lambda$-modules for all $\alpha\in\Aut_\OO(\Lambda)$, then $e\Gamma e$ is an $\OO$-linear crossed product of $e\Lambda e$ and $G$.
	\end{prop}
	\begin{proof}
		Since $e\in \Lambda = \Gamma_1$ is homogeneous, the $\OO$-order $e\Gamma e$ is $G$-graded with $(e\Gamma e)_g=e\Gamma_g e$ for all $g\in G$. In particular $(e\Gamma e)_1 = e\Lambda e$. All we have to show is that $(e\Gamma e)_g$ contains a unit for each $g\in G$. For $g\in G$ let $u_g$ be a (fixed) unit in $\Gamma_g$. Then conjugation by $u_g$ induces an automorphism 
		$\alpha_g$ on $\Gamma_1=\Lambda$, and therefore $u_g e u_g^{-1}=\alpha_g(e)=x_gex_g^{-1}$ for some $x_g\in \UU(\Lambda)$ (as $e\Lambda \cong \alpha_g(e)\Lambda$ implies that $e$ and $\alpha_g(e)$ are conjugate). It follows that $u'_g=u_gx_g^{-1}\in \Gamma_g$ is a unit which commutes with $e$. Now $u'_ge\cdot {u'_g}^{-1}e=e$, that is, $u'_ge\in (e\Gamma e)_g$ is a unit in $e\Gamma e$, as required.
	\end{proof}

	\begin{corollary}\label{corollary crossed prod condens}
		Let $G$ be a finite group and let $\Lambda$ be an $\OO$-order.
		\begin{enumerate}
			\item An $\OO$-linear crossed product $\Gamma$ of $\Lambda^{n\times n}$ and $G$ is isomorphic to $\Gamma'^{n\times n}$, where $\Gamma'$ is an $\OO$-linear crossed product of $\Lambda$ and $G$.
			\item An $\OO$-linear crossed product of $\Lambda$ and $G$ is Morita equivalent to an $\OO$-linear crossed product of $\Lambda_0$ and $G$, where $\Lambda_0$ is the basic order of $\Lambda$.
		\end{enumerate}
	\end{corollary}
	\begin{proof}
		\begin{enumerate}
			\item Let $e_1,\ldots,e_n$ denote the diagonal matrix units in $\Lambda^{n\times n}$ (i.e. the ``standard idempotents''). Since these are conjugate to one another in $\Lambda^{n\times n}$, they are also conjugate in $\Gamma$. As they are orthogonal and sum up to one, it follows that $\Gamma \cong (e_1 \Gamma e_1)^{n\times n}$. All we need to show is that $\Gamma'=e_1\Gamma e_1$ is a crossed product as claimed, and by Proposition~\ref{prop crossed prod condens} this reduces to showing that $e_1\Lambda^{n\times n}\cong \alpha(e_1)\Lambda^{n\times n}$ for any $\alpha\in\Aut_\OO(\Lambda^{n\times n})$. But 
			$$
				\begin{array}{rccccl}
				(e_1\Lambda^{n\times n})^{\oplus n} &\cong& e_1\Lambda^{n\times n} \oplus \ldots \oplus e_n\Lambda^{n\times n}&\cong& \Lambda^{n\times n}&\\&\cong& \alpha(e_1)\Lambda^{n\times n} \oplus \ldots \oplus \alpha(e_n)\Lambda^{n\times n}&\cong& (\alpha(e_1)\Lambda^{n\times n})^{\oplus n}
				\end{array}
			$$
			so our claim follows from the Krull-Schmidt theorem.
			\item We have $\Lambda_0=e\Lambda e$ for an idempotent $e$ with the property that 
			$e\Lambda$ is a basic projective generator for the module category of $\Lambda$, which is unique up to isomorphism. But for any $\alpha\in \Aut_\OO(\Lambda)$ the module $\alpha(e)\Lambda$ is a twist of $e\Lambda$, and therefore also a basic projective generator. That is, $e\Lambda\cong \alpha(e)\Lambda$. Now our claim follows immediately from Proposition~\ref{prop crossed prod condens}. \qedhere
		\end{enumerate}
	\end{proof}
	
	\begin{thm}[{cf. \cite[Theorem in Section 5]{KuelshammerDonovan}}]\label{thm kuelshammer full}
		Let $D$ be a finite $p$-group. If $B$ is a block of $\OO G$ (for some finite group $G$) with defect group isomorphic to $D$, then $B\cong \Gamma^{n\times n}$, where $n\in \N$ and $\Gamma$ is an $\OO$-linear crossed product of $X$ and $B'$ where $X$ and $B'$ are as follows:
		\begin{enumerate}
		\item $X$ is a finite $p'$-group whose order divides $|\Out(D)|^2$
		\item $B'$ is a block of $\OO H$ with defect group $D$, where $H$ is a finite group such that $H=\langle D^h \ | \ h \in H \rangle$. 
		\end{enumerate}
	\end{thm}
	\begin{proof}
		We will go through the relevant steps of K\"ulshammer's argument, and argue wherever necessary why the individual steps work over $\OO$.
		In general we will not go into too much detail about arguments which do not depend on the choice of coefficient ring.
		In particular, it will not be necessary to reprove the claim on the order of $X$.

		In the first step, consider a proper normal subgroup $N\lhd G$. The blocks of $\OO N$ form a single orbit under the action of $G$ (see \cite[Theorem B]{KuelshammerpBlocks}), and we can choose a representative $B'$, together with its stabilizer $G_{B'}\leq G$. Then, by Proposition~\ref{prop fong}, there is a block of $B^*$ of $\OO G_{B'}$ with defect group $D$ such that $B^*\cong B^{[G:G_{B'}]\times [G:G_{B'}]}$. Repeated application of this argument reduces us to the situation where $B$ covers a unique block of $\OO N$ for every normal subgroup $N$ of $G$. 
		
		Now define $H= \langle D^g \ | \ g \in G \rangle $, and let $B_H$ denote the unique block of $\OO H$ covered by $B$. The group $G$ acts
		on $B_H$ be conjugation, and we let $K$ denote the kernel of the homomorphism $G \longrightarrow \Out_k(k\otimes_\OO B_H)$. By $B_K$ we denote the unique block of $\OO K$ covered by $B$. Define $X=G/K$. Note that $k\otimes_\OO B$, $k\otimes_\OO B_H$, $k\otimes_\OO B_K$ are the algebras ``$A$'', ``$B$'' and ``$C$'' from \cite{KuelshammerDonovan}, and the group $X$ is the same as in \cite{KuelshammerDonovan} as well. Hence $X$ is a $p'$-group whose order divides $|\Out(D)|^2$, and the block idempotents $1_{B_G}$ and $1_{B_K}$ are equal (in general, conjugacy of two idempotents in $\OO G$ can already be checked in $kG$; if one of them is central, then the same is true for equality).
		
		By Corollary~\ref{corollary crossed prod} the $\OO$-order $B$ is a crossed product of $X$ and $B_K$. Now $B_K$ and $B_H$ share the same defect group, $H$ is normal in $K$ and $B_H$ is stable under the action of $K$ on $\OO H$ by conjugation. By Corollary~\ref{corollary nat morita} it follows that $B_K$ and $B_H$ are naturally Morita equivalent. In particular $B_K \cong B_H^{n\times n}$ for some $n\in \N$. So now we know that $B\cong (B_H^{n\times n})*_{(\alpha,\gamma)}X$ for some $n$ and some parameter set $(\alpha,\gamma)$. By Corollary~\ref{corollary crossed prod condens} (1) it follows that $B$ is isomorphic to $\Gamma^{n \times n}$, where $\Gamma$ is a crossed product of $B_H$ and $X$. This completes the proof.
%
	\end{proof}

	From Corollary~\ref{corollary crossed prod condens} (2) we get the following statement, which combined with combined Theorem~\ref{thm kuelshammer full} immediately implies Theorem~\ref{thm kuelshammer over o}. 
	\begin{corollary}\label{corollary finite upstairs finite downstairs}
			Let $G$ be a finite group, and let $B$ be a block of $\OO G$ with defect group  $D\leq G$. Let $H$ be a subgroup of $G$ containing $D$ such that
			\begin{enumerate}
			\item $H=\langle D^h \ | \ h\in H \rangle$ 
			\item There is a block $B'$ of $\OO H$ of defect $D$ such that $B$ is Morita equivalent to $\Gamma^{n\times n}$ for some $\OO$-linear crossed product $\Gamma$ of $B'$ and $X$, where $X$ is a finite group whose order divides $|\Out(D)|^2$. 
			\end{enumerate}
			Our claim is: If $B'$ is Morita equivalent to some fixed basic $\OO$-order $\Lambda_0$, then 
			$B$ is Morita equivalent to an $\OO$-linear crossed product of $\Lambda_0$ and $X$ . 
			In particular, if we fix $\Lambda_0$ (that is, if we fix the  Morita equivalence class of $B'$) and let $G$ and $D$ vary, there are only finitely many possibilities for the Morita equivalence class of $B$.
	\end{corollary}
	\begin{proof}
		 By Corollary~\ref{corollary crossed prod condens} (2) the crossed product $\Gamma$ is Morita equivalent to a crossed product of $\Lambda_0$ and $X$, which implies the first part of the statement.
		To prove the second part of the statement, it suffices to realise that if two blocks are Morita equivalent, then their defect groups have the same order (this order can be recovered, for example, as the largest elementary divisor of the Cartan matrix). In other words, there are only finitely many possibilities for $D$ and $X$. Corollary~\ref{corollary finitely many crossed prod} does the rest.
	\end{proof}

	\bibliography{refs}
	\bibliographystyle{plain}	
\end{document}